\documentclass[psamsfonts]{amsart}

\usepackage{amssymb,amsfonts}
\usepackage{amsmath}
\usepackage{mathtools}

\DeclarePairedDelimiter\abs{\lvert}{\rvert}
\usepackage[all,arc]{xy}
\usepackage{enumerate}
\usepackage{mathrsfs}
\usepackage{graphicx}
\usepackage{epsfig}
\usepackage{epstopdf}
\usepackage{tikz-cd}
\usepackage{grffile}
\usepackage{textcomp}
\usepackage[utf8]{inputenc} 
\usepackage[T1]{fontenc}
\graphicspath{ {./Symplectic Geometry/} }

\newtheorem{thm}{Theorem}[section]
\newtheorem{cor}[thm]{Corollary}
\newtheorem{prop}[thm]{Proposition}
\newtheorem{lem}[thm]{Lemma}

\theoremstyle{definition}
\newtheorem{defn}[thm]{Definition}

\newtheorem{exmp}[thm]{Example}

\theoremstyle{remark}
\newtheorem{rem}[thm]{Remark}

\makeatletter
\let\c@equation\c@thm
\makeatother
\numberwithin{equation}{section}

\newcommand{\bbN}{\mathbb{N}}

\newcommand{\bbR}{\mathbb{R}}
\newcommand{\bbZ}{\mathbb{Z}}

\title{Exact Lagrangians in the cotangent bundle of a sphere and a torus}

\author{Raunak Kundagrami \\\today}

\begin{document}

\begin{abstract}
It is known that any closed, exact Lagrangian in the cotangent bundle of a closed, smooth manifold is of the same homotopy type as the zero section. In this paper, we give a Fukaya-theoretic proof of this fact for the sphere and torus to review and demonstrate some of the homological algebra techniques in symplectic geometry.
\end{abstract}

\maketitle

\tableofcontents

\section{Introduction}
There is a well known conjecture in symplectic geometry, called the ``nearby Lagrangian conjecture,'' which states that every closed, exact Lagrangian submanifold of the cotangent bundle of a closed, smooth manifold is Hamiltonian isotopic to the zero section. Though progress has been made toward proving this conjecture, it is still wide open. There are, however, some important partial results (see e.g. \cite{FukayaSeidelSmith}), and we will state one of them here.

\begin{thm}\cite{AbouzaidKragh} Every closed, exact Lagrangian submanifold of the cotangent bundle of a closed, smooth manifold is (simply) homotopy equivalent to the zero section.
\end{thm}

In this paper, we use Fukaya categories and homological algebraic techniques to prove the following result for the cotangent bundle $T^*M$, where $M$ is a sphere $S^2$ or a torus $T^2$.

\begin{thm}(Theorem \ref{sphere} and Theorem \ref{torus}) Let $M$ be $S^2$ or $T^2$. Let $S$ be an orientable, compact, connected, exact Lagrangian in $T^*M$. Then $S$ is topologically $M$.
\end{thm}

To this end, in Chapter 2, we will review the basic definitions of symplectic geometry. In Chapter 3, we will review cellular (co)homology of spaces through examples. In Chapter 4, we will recall the definition of Floer cohomology of two Lagrangians and Fukaya category of symplectic manifolds without going into details. In Chapter 5, we will recall the definition of differential graded (dg) and $A_{\infty}$-categories. We will also present some useful categorical results. Finally, in Chapter 6, we will prove our main theorems.

\subsection{Acknowledgements} This paper is the product of a summer project undertaken by the author when he was a 2nd year undergraduate at the University of Chicago. The project was supervised by Dogancan Karabas. I appreciate his mentoring as well as his numerous suggestions throughout the paper-writing process.

\section{Symplectic Geometry}
Our main reference for this section is \cite{daSilva}.
\subsection{Symplectic Manifolds}

\begin{defn} Let $\omega$ be a differential 2-form on a smooth manifold $M$. We say that $\omega$ is \textbf{symplectic} if $\omega$ is closed (i.e. $d\omega=0$) and if $\omega_p$ is \textit{non-degenerate} for all $p\in M$. In other words, for all $p\in M$, $\omega_p(v,w)=0$ for all $w\in T_p(M)$ implies $v=0$.
\par A \textbf{sympletic manifold} $(M,\omega)$ is a manifold $M$ with symplectic form $\omega$.
\end{defn}

\begin{rem}
If $(M,\omega)$ is a symplectic manifold, then $\dim M$ is even.
\end{rem}

\begin{defn}
Let $(M,\omega_1)$ and $(N,\omega_2)$ be symplectic manifolds. Then, a diffeomorphism $f:M\to N$ is a \textbf{symplectomorphism} if $f^*\omega_2=\omega_1$ where $f^*$ is the pullback of $f$.
\end{defn}

We will now present 2 examples of symplectic manifolds.
\begin{exmp}\label{R2n}
Let $M=\bbR^{2n}$ with coordinates $x_1,x_2,\dots,x_n,y_1,\dots,y_n$ and let $\omega=\sum_{i=1}^n dx_i\wedge dy_i$. We can check that $\omega$ is a symplectic form.
\par We first check closure. We note that
$$d\omega=d(\sum_{i=1}^n dx_i\wedge dy_i)=\sum_{i=1}^n d(dx_i\wedge dy_i)=\sum_{i=1}^n(0\wedge dy_i)-(dx_i\wedge 0)=0$$
and therefore $\omega$ is closed. We now check non-degeneracy. Let $p\in\bbR^n$, and write $v=\sum_{i=1}^n v_1^i \frac{\partial}{\partial x_i} + v_2^i \frac{\partial}{\partial y_i}$ and $w=\sum_{i=1}^n w_1^i \frac{\partial}{\partial x_i} + w_2^i \frac{\partial}{\partial y_i}$ for some $v_j^i,w_j^i\in\bbR$. Suppose
$$\omega_p(v,w)=\sum_{i=1}^n v_1^i w_2^i - v_2^i w_1^i=0$$
for all $w\in T_p(M)$. Then, since $w_1,w_2$ are abitrary, we must have $v=0$ as desired. Therefore $w_p$ is non-degenerate for all $p\in\bbR$, and $\omega$ is thus symplectic.
\end{exmp}

\begin{exmp} The next example we will consider is the \textit{cotangent bundle}. Let $M$ be an $n-$dimensional manifold and let $T^*M$ be its cotangent bundle. Let $M$ have charts of the form $(U,x_1,x_2,\dots,x_n)$ where $x_i:U\to\bbR$ for each $i$. Then for any $x\in U$, we have a basis of $T_x^*M$ consisting of the differentials $(dx_1)_x,\dots,(dx_n)_x$. In other words, each $\xi\in T_x^* M$ can be written as $\xi=\sum_{i=1}^n \xi_i(dx_i)_x$ for some coefficients $\xi_i\in\bbR$. This observation allows us to construct a chart on $T^*M$ of the form
$$(T^*U,x_1,\dots,x_n,\xi_1,\dots,\xi_n)$$
This chart gives the cotangent coordinates on $T^*U\subset T^*M$ associated to the coordinates on $U\subset M$. To show that the cotangent coordinates indeed form a chart, it remains to show that their transition maps on overlapping $T^*U,T^*U'\subset T^*M$ are smooth. Indeed, for two charts on $T^*M$, $(T^*U,x_1,\dots,x_n,\xi_1,\dots,\xi_n)$ and $(T^*U',x'_1,\dots,x'_n,\xi'_1,\dots,\xi'_n)$, if $(x,\xi)\in T^*U\cap T^*U'$ we have
$$\xi=\sum_{i=1}^n \xi_i(dx_i)_x=\sum_{i,j=1}^n \xi_i(\frac{\partial x_i}{\partial x_j'})_x (dx_j')_x$$
where the transition map $\xi_j'=\sum_{i=1}^n \xi_i (\frac{\partial x_i}{\partial x_j'})_x$ is smooth. Additionally, the transition map $x'_i$ is smooth since $M$ is a smooth manifold, and we have thus shown that $T^*M$ is a $2n-$dimensional manifold. We can now introduce its symplectic structure.

\par For $T^*U\subset T^*M$ we can define a symplectic form $\omega$ similarly to how we did on $\bbR^n$. Namely, 
$$\omega=\sum_{i=1}^n dx_i\wedge d\xi_i$$
 This 2-form is symplectic for the same reasons desribed in Example $\ref{R2n}$. We must check, however, that $\omega$ is coordinate-independent. Firstly we note that we can define
$$\alpha=-\sum_{i=1}^n \xi_i dx_i$$
and observe that $\omega=d\alpha$. If $\alpha$ is coordinate-independent, then so too is $\omega$.
\par To show that $\alpha$ is coordinate-independent, we consider two cotangent coordinate charts $(T^*U,x_1,\dots,x_n,\xi_1,\dots,x_n)$ and $(T^*U',x_1',\dots,x_n',\xi_1',\dots,\xi_n')$. Then, on $T^*U\cap T^*U'$ we have the transition map $\xi_j'=\sum_{i=1}^n\xi_i (\frac{\partial x_i}{\partial x_j'})$. We also have $dx_i=\sum_{i=1}^n(\frac{\partial x_i}{\partial x'_j}) dx'_j$, and therefore

\[\alpha'=-\sum_{j=1}^n \xi'_j dx'_j=-\sum_{j=1}^n \sum_{i=1}^n \xi_i (\frac{\partial x_i}{\partial x'_j}) dx'_j\] $$=-\sum_{i=1}^n \xi_i(\sum_{j=1}^n (\frac{\partial x_i}{\partial x'_j})dx'_j)=-\sum_{i=1}^n \xi_i dx_i=\alpha$$

Therefore $\alpha$ is coordinate independent, and so too is $\omega$. Thus $T^*M$ is a symplectic manifold.
\par We refer to $\omega$ as the \textbf{canonical symplectic form} on $T^*M$ and $\alpha$ as the \textbf{tautological 1-form}. The existence of this tautological form motivates a new definition.
\end{exmp}

\begin{defn}
Let $(M,\omega)$ be a symplectic manifold. We say that $M$ is an \textbf{exact symplectic manifold} if there is some one-form $\alpha$ on $M$ such that $\omega=d\alpha$. We call $\alpha$ a \textbf{Liouville 1-form}.
\end{defn}

We note that the cotangent bundle is an exact symplectic manifold.

\subsection{Lagrangians}

\begin{defn} Let $L$ be a submanifold of a $2n-$dimensional symplectic manifold $(M,\omega)$. We call $L$ a $\textbf{Lagrangian submanifold}$ of $M$ is $\dim L=\frac{1}{2}\dim M=n$ and if the symplectic form $\omega$ vanishes on $L$. In other words, if $i:L\hookrightarrow M$ is the inclusion map, then the pullback $i^*\omega=0$.
\end{defn}

Just as we have defined exact symplectic manifolds, there is a related notion for Lagrangians.

\begin{defn}
Let $M$ be an exact symplectic manifold with Liouville 1-form $\alpha$. Let $L$ be a Lagrangian. We say that $L$ is \textbf{exact} if $\alpha|_L=df$ for some function $f:L\to\bbR$.
\end{defn}

\begin{rem}\label{alphaomegavanish}
If $\alpha|_L=df$, then $\omega|_L=0$.
\end{rem}

We will now provide two examples of exact Lagrangian submanifolds of the cotangent bundle. Let $M$ be an $n-$dimensional manifold. Then $T^*M$ is a $2n-$dimensional exact symplectic manifold with the tautological 1-form $\alpha$.

\begin{exmp}
We define the \textbf{zero section} $M_0$ of $T^*M$
$$M_0=\{(x,\xi)\in T^*M\mid \xi=0\text{ in } T_x^*M\}$$
We note first that this is an $n-$dimensional submanifold with charts $(U,x_1,\dots,x_n)$ inherited from $M$. Moreover, for $U\subset M$, we have that $M_0\cap T^* U$ is given by the equations $\xi_1,\dots,\xi_n=0$. As a result, the tautological form $\alpha=\sum\xi_i dx_i$ vanishes on $M_0\cap T^*U$. Then, if $i:M_0\hookrightarrow T^*M$ is the inclusion map, $i^*\alpha=0$ and therefore the zero section is an exact Lagrangian.
\end{exmp}

\begin{exmp}
Let $x\in M$, and let $F_x=\{(x,\xi)\mid \xi\in T_x^* M\}$ be a cotangent fibre of $M$. Then $F_x$ is an $n-$dimensional manifold with coordinate chart $(F_x,\xi_1,\dots,\xi_n)$ inherited from the cotangent coordinate chart of $T^*M$. We will now show that the tautological 1-form $\alpha$ vanishes on $F_x$. Because the $x_1,\dots,x_n$ are fixed by $x$, the differentials $dx_i$ vanish, and thus so does the tautological form $\alpha=\sum \xi_i dx_i$. Then $i^*\alpha=0$ similarly to the previous example. Therefore $F_x$ is an exact Lagrangian submanifold of $T^*M$.
\end{exmp}

\section{Homology and Cohomology}
We will mostly follow \cite{Hatcher} in this section.
\subsection{Cellular Homology}
This section will explain how to calculate the homologies of a manifold by using a handle-body decomposition. The calculation of the homologies will be explained using the circle $S^1$ as an example. Further examples will be presented at the end of the section.

\begin{defn} Let $M$ be an $n-$dimensional manifold, and let $D^i$ be the $i-$dimensional closed disk. Then, $M$ can be built by attaching various $i-$handles $D^i\times D^{n-i}$ for $0\leq i\leq n$. Each $i-$handle is attached along $\partial D^{i}\times D^{n-i}=S^{i-1}\times D^{n-i}$. Handles are attached in non-decreasing fashion (in terms of index) and each handle is attached to the boundary of the union of the previously-attached handles. The various $i-$handles needed to build $M$ is referred to as the manifold's \textbf{handlebody decomposition}.
\end{defn}

\begin{defn}
A \textbf{chain complex} is a collection of vector spaces and linear maps between them:
$$\dots V_2\xrightarrow{d_2} V_1\xrightarrow{d_1} V_0\xrightarrow{d_0} V_{-1}\dots$$
such that $d_i d_{i+1}=0$.
\end{defn}

In this section, we will use real-coefficient vector spaces.

\begin{exmp}
Let $M=S^1$. Then $\dim M=1$. As shown in the Figure \ref{s1handle}, $S^1$ can be decomposed into one $0-$handle $D^0\times D^1$ and one $1-$handle $D^1\times D^0$, attached to the boundary of the zero handle along the two points that make up $\partial D^1\times D^0$.

\begin{figure}[h]
\includegraphics[scale=0.35]{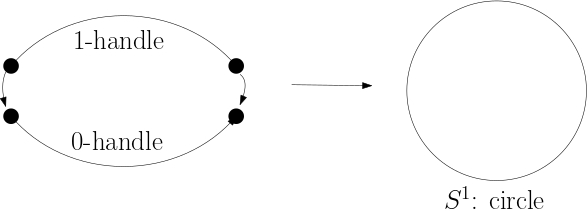}
\caption{A handlebody decomposition of the circle $S^1$}
\label{s1handle}
\end{figure}

\par From this handlebody decomposition, we can build a chain complex, where $\dim V_i=\text{number of i-handles}$. Thus, for this example, $V_0\simeq\bbR$ and $V_1\simeq\bbR$. We also note that $V_i=0$ for all $i\neq 0,1$. Thus our chain complex for $S^1$ looks like this so far
$$\dots 0\xrightarrow{d_2}\bbR\xrightarrow{d_1}\bbR\xrightarrow{d_0} 0\dots$$
and it remains to determine the maps $d_i$.

\begin{defn}
The maps $d_i:V_i\to V_{i-1}$ are called \textbf{boundary homomorphisms}. We will explain how to determine these maps through the example of $S^1$.
\end{defn}

We will calculate $d_1:V_1=\bbR\to V_0=\bbR$. Let $a$ be a basis for $V_1$ and $b$ be a basis for $V_0$. Since $d_1$ is linear, it is sufficient to understand the effect of the map on the basis elements of the vector spaces. Thus we can write $d_1$ in the form
$$d_1:a\to n_{ab} b$$
It remains to calculate $n_{ab}$. Since $d_1$ is the attaching map between the $1-$handle and the $0-$handle, we must consider the former is attached to the latter. In particular, we must consider the handles without their thickening.

\begin{defn} An $i-$handle $D^i \times D^{n-i}$ is the \textbf{thickening} of $D^i$. $D^i$ is referred to as an \textbf{$i-$cell}.
\end{defn}

We depict $S^1$'s cell decomposition (handlebody decomposition without thickening) in Figure \ref{s1cell}. To calculate the map $d_1$, we first assign orientations to the handles as shown below. 

\begin{figure}[h]
\includegraphics[scale=0.45]{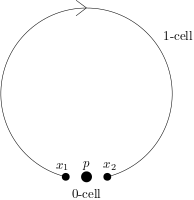}
\caption{A 	cell decomposition of the circle $S^1$.}
\label{s1cell}
\end{figure}

Next we, define $h:\partial D^1\to D^0$ be the attaching map, and let $p\in D^0$. Then $n_{ab}$ is defined as the "signed" count of the elements of $h^{-1}(p)$. In this case, $h^{-1}(p)$ consists of two points $x_1,x_2\in D^1$. The point $x_1$ is the starting point of $D^1$ in the given orientation and so it contributes a $-1$ sign to $n_{ab}$, while $x_2$ is the endpoint of $D^1$ and contributes at $+1$ sign to $n_{ab}$. Therefore $n_{ab}=0$ which tells us that
$$d_1:a\to 0\cdot b$$
is the zero map. Moreover,
$$d_0: \bbR\to 0$$
$$d_2: 0\to\bbR$$
are both linear homomorphisms and are thus necessarily the zero map. We can update our chain complex
$$\dots 0\xrightarrow{0}\bbR\xrightarrow{0}\bbR\xrightarrow{0} 0\dots$$
and now have the necessary information to calculate the homologies.

\begin{defn} For a manifold $M$, the $n-$th homology $(n\in\bbZ)$ with real coefficients $H_n(M,\bbR)$ is defined by $H_n(M,\bbR)=\frac{\text{Ker}(d_n)}{\text{Image}(d_{n+1})}$.
\end{defn}

The following is well known.

\begin{thm}
$H_n(M,\bbR)$ is an invariant of the topological space $M$.
\end{thm}

Thus we have:
$$H_0(S^1,\bbR)=\frac{\text{Ker}(d_0)}{\text{Image}(d_1)}=\bbR/0=\bbR$$
$$H_1(S^1,\bbR)=\frac{\text{Ker}(d_1)}{\text{Image}(d_2)}=\bbR/0=\bbR$$

The rest of the homologies $H_n(S^1,\bbR)$ for $n\neq 0,1$ are zero, and can be covered with the following proposition.
\end{exmp}

\begin{prop}\label{0 homology} Let $M$ be an $n-$dimensional manifold. Then, for $i>n$ or $i< 0$, $H_i(M,\bbR)=0$
\end{prop}

\begin{proof}
Since $M$ is $n-$dimensional, there are no $i-$handles for $i>n$. Similarly, there are no $i-$handles for $i<0$, so the relevant chain complex takes the following form:
$$\dots0\xrightarrow{d_{n+2}} 0\xrightarrow{d_{n+1}} V_n\xrightarrow{d_n}\dots V_0\xrightarrow{d_0} 0\xrightarrow{d_{-1}}0\dots$$

Suppose $i>n$. Then $d_i: 0\to V_{i-1}$ is necessarily the zero map, as is $d_{i+1}:0\to 0$
$$H_i(M,\bbR)=\frac{\text{Ker}(d_i)}{\text{Image}(d_{i+1})}=0/0=0$$
The proof is similar for the case where $i< 0$.
\end{proof}

Thus we have calculated the homologies of the circle. We will use the sphere $S^2$ and the torus $T$ as second and third examples.

\begin{exmp} Let $S^2$ be a (two-dimensional) sphere. Then its handlebody decomposition consists of one 0-handle ($D^0\times D^2$) and one 2-handle ($D^2\times D^0$). Our chain complex takes the following form
$$\dots 0\rightarrow V_2 \xrightarrow{d_2}V_1 \xrightarrow{d_1} V_0\xrightarrow{d_0} 0\dots$$
$$\dots 0 \xrightarrow{d_3}\bbR\xrightarrow{d_2} 0 \xrightarrow{d_1}\bbR\xrightarrow{d_0} 0\dots$$

In this case, we can observe (without calculation) that both $d_2,d_0:\bbR\to 0$ are the zero map and that both $d_1,d_3$ are the identity map on domain $\{0\}$. Thus we get our homologies:

$$H_0(S^2,\bbR)=\frac{\text{Ker}(d_0)}{\text{Image}(d_1)}=\bbR/0=\bbR$$
$$H_1(S^2,\bbR)=\frac{\text{Ker}(d_1)}{\text{Image}(d_2)}=0/0=0$$
$$H_2(S^2,\bbR)=\frac{\text{Ker}(d_2)}{\text{Image}(d_3)}=\bbR/0=\bbR$$
\end{exmp}

\begin{exmp}
Let $T$ be a (two-dimensional) torus. Its handlebody decomposition consists of one 0-handle ($D^0\times D^2$), two 1-handles ($D^1\times D^1$), and one 2-handle ($D^2\times D^0$) as shown in the Figure \ref{t2handle}.

\begin{figure}[h]
\includegraphics[scale=0.45]{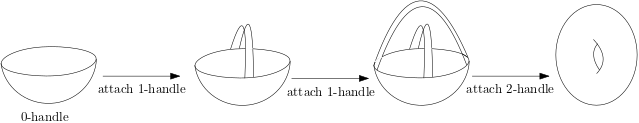}
\caption{A handlebody decomposition of the torus $T^2$.}
\label{t2handle}
\end{figure}

\par Thus we have the following chain complex
$$\dots 0\rightarrow V_2 \xrightarrow{d_2}V_1 \xrightarrow{d_1} V_0\xrightarrow{d_0} 0\dots$$
$$\dots 0 \rightarrow\bbR\xrightarrow{d_2}\bbR^2\xrightarrow{d_1}\bbR\xrightarrow{d_0} 0\dots$$

Let $V_2$ be generated by basis $a$, $V_1$ by $b,c$ and $V_0$ by $d$. We then have the cell-decomposition (no thickening) with assigned orientations as shown in Figure \ref{t2cell}.

\begin{figure}[h]
\includegraphics[scale=0.35]{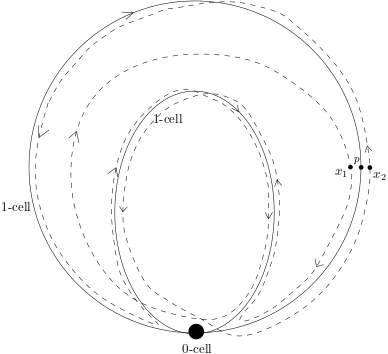}
\caption{A cell decomposition of the torus $T^2$. The 2-handle is attached to the 1-cells and 0-cell along the dotted line.}
\label{t2cell}
\end{figure}

\par We will start by calculating $d_1$. We need determine the following coefficients
$$b\to n_{bd} d$$
$$c\to n_{cd} d$$
We can observe however that both these attachments are equivalent to our circle example. Thus $n_{bd}=n_{cd}=0$, and hence, $d_1=0$. It remains to calculate $d_2$. For this map, we need two coefficients
$$a\to n_{ab} b + n_{ac} c$$

We start with $n_{ab}$. Picking a point $p$ in the  1-cell generated by $b$, we see that for the attaching map $h$, $h^{-1}(p)$ has two points ($x_1,x_2$). At $x_1$, the orientation of the 1-cell and the orientation of the boundary of the 2-cell are the same, so $x_1$ has sign $+1$. At $x_2$, the two orientations are opposite so $x_2$ has sign $-1$. Then $n_{ab}=1-1=0$. Similarly $n_{ac}=0$, so $d_2=0$. We can then finish our chain complex
$$\dots\bbR\xrightarrow{0}\bbR^2\xrightarrow{0}\bbR\xrightarrow{0} 0\dots$$
and calculate the homologies
$$H_0(T,\bbR)=\bbR$$
$$H_1(T,\bbR)=\bbR^2$$
$$H_2(T,\bbR)=\bbR$$

\end{exmp}

\begin{defn} For $g\geq 1$, the genus $g$-surface is defined as the connected sum $\Sigma_g:=T\#T\#T\dots\#T \;\;\; \text{($g$ times)}$. A genus 0 surface is a sphere.
\end{defn}

The calculation of the homologies of $\Sigma_g$ follows nicely from the homologies of the torus $T$, as can be seen in the following example.

\begin{exmp}\label{genusg}
We first note that the cell decomposition of $\Sigma_g$ is
\begin{itemize}
\item 1 two-cell
\item 2g one-cells
\item 1 zero-cell
\end{itemize}
Hence our cell complex is
$$0\xrightarrow{d_3}\bbR\xrightarrow{d_2}\bbR^{2g}\xrightarrow{d_1}\bbR\xrightarrow{d_0} 0$$
We note then that $d_0$ is the zero map and also that \textit{im}($d_3$)=0. We now determine the rest of the boundary homomorphisms.
\par Let the 0-cell be generated by $a$ and let the $1-cells$ by generated by $b_1,\dots,b_{2g}$. Each of these attaching maps $b_i\to n_{{b_i} a} a$ is similar to the 1-torus case and thus $n_{{b_i} a}=0$ for all $1\leq i\leq 2g$. Therefore $d_1$ is the zero map.
\par Let the 2-cell be generated by $c$. Then each attaching map $c\to n_{c {b_i}} b_i$ is similar to the 1-torus case and thus $n_{c _{b_i}}=0$ for all $1\leq i\leq 2g$. Therefore $d_2=0$. We can now calculate the homologies of $\Sigma_g$.
$$H_0(\Sigma_g,\bbR)=\textit{ker}(d_0)/\textit{im}(d_1)=\bbR/0=\bbR$$
$$H_1(\Sigma_g,\bbR)=\textit{ker}(d_1)/\textit{im}(d_2)=\bbR^{2g}/0=\bbR^{2g}$$
$$H_2(\Sigma_g,\bbR)=\textit{ker}(d_2)/\textit{im}(d_3)=\bbR/0=\bbR$$
\end{exmp}

We have the following classical result.

\begin{thm}
An orientable, connected, topological 2-dimensional manifold is a genus $g$-surface for some $g\geq 0$.
\end{thm}

In this instance, the homologies are all different, and thus these homologies give us one way to classify such a surface. However, another way to classify the surface is via the \textbf{Euler characteristic}.

\begin{defn}
Let $M$ be a manifold. We recall then that $\dim V_i$ is the number of $i-$cells in the cell decomposition of M. We define the Euler characteristic $\chi$ of $M$ as $\chi(M)=\sum_{i\in\bbZ} (-1)^i \dim V_i$.
\end{defn}

\begin{rem}
The Euler characteristic is a priori not well defined, but we will show it is well defined.
\end{rem}

We can show that the Euler characteristic is a homology invariant via the following theorem.

\begin{thm}\label{EulerCharRewrite} Let $M$ be a manifold. Then $\chi(M)=\sum_{i\in\bbZ} (-1)^i \dim H_i(M,\bbR)$.
\end{thm}

\begin{proof}
Consider an arbitrary cell decomposition of M and its associated chain complex
$$0\rightarrow V_n\xrightarrow{d_n}\dots\xrightarrow{d_2} V_1\xrightarrow{d_1} V_0\xrightarrow 0\dots$$
and define $Z_i=\text{Ker}(d_i)$, $B_i= \text{Im}( d_{i+1})$, $H_i=H_i(M,\bbR)$. Then the following are short exact sequences
$$0\rightarrow Z_i\rightarrow V_i\rightarrow B_{i-1}\rightarrow 0$$
$$0\rightarrow B_i\rightarrow Z_i\rightarrow H_i\rightarrow 0$$
By Rank-Nullity, we have
$$\dim V_i=\dim Z_i + \dim B_{i-1}$$
$$\dim Z_i=\dim B_i + \dim H_i$$
By substituting the second equation into the first, we get
\begin{center}$\dim V_i=\dim B_i + \dim B_{i-1} +\dim H_i\implies (-1)^i \dim V_i=(-1)^i(\dim B_i +\dim B_{i-1} +\dim H_i)\implies \sum_{i\in\bbZ} (-1)^i \dim V_i=\sum_{i\in\bbZ} (-1)^i \dim H_i$ \end{center}
as desired.
\end{proof}

\begin{cor}
$\chi(M)$ is well-defined.
\end{cor}

\begin{proof}
Let $M$ be a manifold and let $\chi_1$ and $\chi_2$ be Euler characteristics for two different cell decompositions of $M$. Then, by Theorem \ref{EulerCharRewrite},
$$\chi_1(M)=\sum_{i\in\bbZ} (-1)^i \dim H_i(M,\bbR)=\chi_2(M)$$
since homology is an invariant of $M$.
\end{proof}

\begin{rem} By our remarks in the proof of Proposition \ref{0 homology}, the above definition simplifies to $\chi(M)=\sum_{i=0}^n (-1)^i \dim V_i$ if $M$ is $n-$dimensional.
\end{rem}

\begin{prop}
If $S$ is a genus $g$-surface, then $\chi(S)=2-2g$.
\end{prop}

\begin{proof} This follows from the cell decomposition shown in Example \ref{genusg}.
\end{proof}

Thus we can use the Euler characteristic to uniquely identify orientable surfaces.

\begin{defn}
If $C_1$ and $C_2$ are two chain complexes with
$$C_1: \dots\rightarrow V_i\xrightarrow{d_i}\dots$$
$$C_2: \dots\rightarrow W_i\xrightarrow{e_i}\dots$$
then $C_1\oplus C_2$ is a chain complex
$$C_1\oplus C_2: \dots\rightarrow V_i\oplus W_i\xrightarrow{\begin{pmatrix} d_i & 0\\ 0 &   e_i\end{pmatrix}}\dots$$
\par
Also, if $C$ is a chain complex
$$C: \dots\rightarrow V_i\xrightarrow{d_i}\dots$$
then $C[n]$ is a chain complex
$$C[n]: \dots\rightarrow W_i\xrightarrow{e_i}\dots$$
such that $W_i=V_{i+n}$ and $e_i=(-1)^n d_{i+n}$.
\end{defn}

The Euler characteristic can also be defined directly for bounded chain complexes in a natural way. This gives rise to a few useful properties of the Euler characteristic.

\begin{prop}\label{euler props}
\begin{enumerate}
\item[a)] Given two bounded chain complexes $C$ and $D$, $\chi(C\oplus D)=\chi(C)+\chi(D)$
\item[b)] Given a bounded chain complex $C[s]$ with $s\in\bbZ$, $\chi(C[s])=(-1)^{s} \chi(C)$
\end{enumerate}
\end{prop}

\begin{proof}
\begin{enumerate}
\item[a)] Let $C$ consist of vector spaces $V_i^1$ and $D$ consist of vector spaces $V_i^2$. Then
\begin{center}$
\chi(C\oplus D)=\sum_{i\in\bbZ} (-1)^i \dim(V_i^1\oplus V_i^2)=\sum_{i\in\bbZ} (-1)^i(\dim V_i^1+\dim V_i^2)=\sum_{i\in\bbZ} (-1)^i \dim V_i^1 + \sum_{j\in\bbZ} (-1)^j V_1^2=\chi(C)+\chi(D)
$\end{center}

\item[b)] Let $C$ consist of vector spaces $V_i$. Then $C[s]$ consists of vector spaces $W_i=V_{i+s}$. We can then write
\begin{center}$\chi(C[s])=\sum_{i\in\bbZ} (-1)^{i} \dim W_{i}=\sum_{i\in\bbZ} (-1)^{i} \dim{V_{i+s}}=\sum_{j\in\bbZ}(-1)^{j-s}\dim{V_j}=(-1)^s\sum_{j\in\bbZ}(-1)^j \dim{V_j}=(-1)^{s} \chi(C)$\end{center}
as desired.
\end{enumerate}
\end{proof}

\begin{defn}
A \textbf{cochain complex} is similar to a chain complex but with the differential maps increasing the degree by $1$ rather than decreasing the degree by $1$:
$$\dots\rightarrow V^{-1}\xrightarrow{d^{-1}} V^0\xrightarrow{d^0} V^1\xrightarrow{d^1} V^2\rightarrow\dots$$
The homology of a cochain complex is referred to as its \textbf{cohomology}. We can more specifically define cochain complexes and cohomology in the case of a topological space.
\end{defn}

\begin{defn}
Let $M$ be a topological space with an associated chain complex
$$\dots V_2\xrightarrow{d_2} V_1\xrightarrow{d_1} V_0\xrightarrow{d_0} V_{-1}\dots$$ 
Then, define $V^i$ to be the dual groups $\text{Hom}(V_i,\bbR)$, and define the maps $d^i:V^i\to V^{i+1}$ to be the dual maps of the boundary homomorphisms $d_i:V_i\to V_{i-1}$. We then have our cochain complex
$$\dots\rightarrow V^{-1}\xrightarrow{d^{-1}} V^0\xrightarrow{d^0} V^1\xrightarrow{d^1} V^2\rightarrow\dots$$
and define the $n$-th cohomology as $H^n(M,\bbR)=\frac{\text{Ker}(d^n)}{\text{Image}(d^{n-1})}$.
\end{defn}

\begin{rem}
By the Universal Coefficient Theorem, $H^n(M,\bbR)=H_n(M,\bbR)$ if $H_n(M,\bbR)$ is finite-dimensional. This is in particular true because we are working with field coefficients. Therefore, the results of this chapter involving homology are also true for cohomology.
\end{rem}

\section{Floer Cohomology}
We won't give rigorous definitions in this section. See \cite{Auroux} for more information, or for a more rigorous treatment, see \cite{Seidel}
\subsection{Floer Cohomology of Compact Lagrangians}.

\begin{defn}
Let $M$ be an exact symplectic manifold. If $L_1,L_2\subset M$ are compact, exact, oriented Lagrangians, then we can calculate their \textbf{Floer cohomology} $HF^*(L_1,L_2)$ by using the intersection points of $L_1$ and $L_2$ to construct a cochain complex (assuming $L_1$ and $L_2$ intersect transversally and at finitely many points). In this section, we will use $\bbZ_2$-graded chain complexes of the following form:
$$\dots CF^0(L_1,L_2)\xrightarrow{d^0} CF^1(L_1,L_2)\xrightarrow{d^1} CF^0(L_1,L_2)\dots$$

The vector spaces $CF^i(L_1,L_2)$ will be $\bbR$-vector spaces with basis elements of the form $x\in L_1\cap L_2$ such that $\text{deg}(x)=i$. When $M$ is 2-dimensional, the degree of any given basis element $x\in L_1\cap L_2$ is determined by the right-hand rule. Starting at a given $x\in L_1\cap L_2$, we orient our palm in the direction of $L_1$ and let our fingers curl towards $L_2$. If our thumb points into the screen, then $x$ is assigned an even parity (degree 0) and is a basis element for $CF^0(L_1,L_2)$. If our thumb points out of the screen, then $x$ is assigned an odd parity (degree 1) and is a basis element for $CF^1(L_1,L_2)$. See Figure \ref{righthandrule} for the two possible parity scenarios (up to isotopy).

\begin{figure}[h]
\includegraphics[scale=0.45]{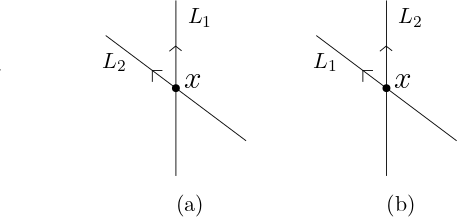}
\caption{The two possible parity scenarios in Floer cohomology: (a) odd parity, (b) even parity}
\label{righthandrule}
\end{figure}

\par We now describe how to determine the differential maps in our cochain complex. The maps will take the form
$$d^i:CF^i(L_1,L_2)\to CF^{i+1}(L_1,L_2)$$
$$x\to\sum_y n_{xy} y$$
where the sum is over $y\in L_1\cap L_2$ of deree $i+1$ and $n_{xy}$ is the (signed) count of the pseudoholomorphic disks of the form shown in Figure \ref{holdisk} (see \cite{Auroux}). By the Riemann Mapping Theorem, the pseudoholomorphic disks can be thought of as usual holomorphic disks if $M$ is 2-dimensional.
\end{defn}

\begin{figure}
\includegraphics[scale=0.45]{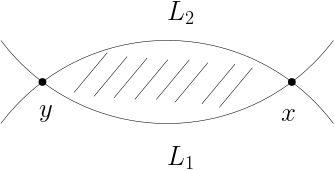}
\caption{A pseudoholomorphic disk used in the determination of Floer cohomology.}
\label{holdisk}
\end{figure}

With the vector spaces and maps of the cochain complex determined, Floer cohomology is just the cohomology of this particular cochain complex. The following gives an example calculation of Floer cohomology.

\begin{exmp}\label{zero section HF} Let $M$ be the cotangent bundle $T^*S^1$, let $L_1$ be as in Figure \ref{cylinder w lagrangians}, and let $L_2$ be its zero section. We let $L_1$ intersect $L_2$ at two points $x$ and $y$ as shown in Figure \ref{cylinder w lagrangians}.
\end{exmp}

\begin{figure}[h]
\includegraphics[scale=0.45]{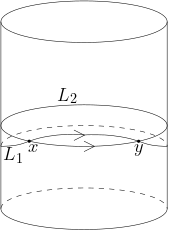}
\caption{Two Lagrangians in $T^*S^1$}
\label{cylinder w lagrangians}
\end{figure}

We first determine the degrees of $x,y\in L_1\cap L_2$. Consider the point $x$, and use the right hand rule. Starting at $x$, we orient our palm in the direction of $L_1$ and let our fingers curl towards $L_2$. By doing so, we see that our thumb points into the screen, which corresponds to an even parity. By considering the right hand rule for the intersection point $y$, we see that it has odd parity (thumb pointing out of the screen). Thus  $CF^0(L_1,L_2)$ has one basis element $x$ and $CF^1(L_1,L_2)$ also has one basis element $y$. So both are one-dimensional $\bbR$-vector spaces. In other words, we have the following cochain complex:
$$\dots\bbR\xrightarrow{d^0}\bbR\xrightarrow{d^1}\bbR\dots$$
We must now calculate the maps
$$d^0:x\to n_{xy} y$$
$$d^1:y\to n_{yx} x$$
by determining $n_{xy}$ and $n_{yx}$.

\par We note that there are two relevant disks bounded by $L_1$ and $L_2$, and that both correspond to $n_{xy}$. One contributes a sign of $+1$ to $n_{xy}$ and the other a sign of $-1$, and hence $n_{xy}=0$. We will not explain how we determine signs here. Similarly, because there are no other disks, $n_{yx}=0$, so our cochain complex is:
$$\bbR\xrightarrow{0}\bbR\xrightarrow{0}\bbR$$
We then calculate the cohomologies in the same way as we did for the usual cohomology:
$$HF^0(L_1,L_2)=\frac{\text{Ker} (d^0)}{\text{Image} (d^1)}=\bbR/0=\bbR$$
$$HF^1(L_1,L_2)=\frac{\text{Ker} (d^1)}{\text{Image} (d^0)}=\bbR/0=\bbR$$

\begin{defn}
Let $(M,\omega)$ be a symplectic manifold and let $H_t:M\to\bbR$ be a family of smooth functions (called \textbf{Hamiltonian}) for $t\in [0,1]$. Then, by non-degeneracy, there is a unique vector field $X_{H_t}$ (called a \textbf{Hamiltonian vector field}) on $M$ such that $\omega(\cdot,X_{H_t})=dH_t$ for each $t\in[0,1]$. Then, there is a family of symplectomorphisms $\rho_t: M \to M$ that is generated by $X_{H_t}$:
$$\begin{cases}\rho_0=\text{id} \\ \frac{d\rho_t}{dt}\circ\rho_t^{-1}=X_{H_t} \end{cases}$$

If there are two submanifolds $A$ and $B$ such that $A=\rho_1(B)$, then we say that $A$ is \textbf{Hamiltonian isotopic} to $B$, and we call $\rho_t$ a \textbf{Hamiltonian isotopy}.
\end{defn}

$L_1$ in Example \ref{zero section HF} can be thought of as if it is obtained by a Hamiltonian isotopy of $L_2$. We can consider Lagrangians $L_1$ and $L_2$ as ``the same" when one is obtained by a Hamiltonian isotopy of the other. This motivates a definition for Floer cohomology that works even if $L_1=L_2$ or if $L_1$ and $L_2$ intersect at infinitely many points or intersect non-transversally.

\begin{defn} Let $L_1,L_2\subset M$ be compact, exact, oriented Lagrangians. Then $HF^*(L_1,L_2):=HF^*(\phi_{1}(L_1),L_2)$, for some Hamiltonian isotopy $\phi_t$ such that $\phi_{1}(L_1)$ and $L_2$ intersect transversally and at finitely many points.
\end{defn}

\begin{rem}
$HF^*(L_1,L_2)$ is invariant under Hamiltonian isotopy of Lagrangians, and hence the above definition is well-defined.
\end{rem}

\begin{thm}\label{floer is homology} If $L$ is a compact, connected, oriented Lagrangian, then $HF^*(L,L)\simeq H^*(L)$, where $H^*(L)$ is the cohomology of $L$.
\end{thm}
 
Consider $L_2$ from Example \ref{zero section HF}. Since the zero section $L_2$ is Hamiltonian isotopic to $S^1$, we can conclude via the previous theorem that
$$H^0(S^1)=\bbR$$
$$H^1(S^1)=\bbR$$
which is what we expect.

\subsection{Floer Cohomology of Non-Compact Lagrangians}

\begin{defn} Let $(M,\omega)$ be an exact symplectic manifold with Liouville 1-form $\alpha$. Then, by non-degeneracy, there is a unique vector field $Z$ on $M$ such that $\iota_Z \omega=\alpha$. This vector field $Z$ is called the \textbf{Liouville vector field}.
\par We call $M$ a \textbf{Liouville manifold} if its Liouville vector field $Z$ is complete and outward pointing at infinity.
\end{defn}

\begin{defn}
Let $(M,\omega)$ be a Liouville manifold with Liouville vector field $Z$. Then $M$ can be thought of as $M=\overline{M}\cup (\partial\overline{M}\times[1,\infty))$, where $\overline{M}\subset M$ is compact such that its boundary $\partial\overline{M}$ intersects $Z$ outwardly transversally. We call $\overline{M}$ the \textbf{Liouville domain} and $\partial\overline{M}\times[1,\infty)$ the \textbf{cylindrical end}. We refer to $r\in[1,\infty)$ as the \textbf{radial coordinate} for the cylindrical end.
\end{defn}

\begin{defn}
Let $M$ be a Liouville manifold with Liouville 1-form $\alpha$. A Lagrangian $L\subset M$ is \textbf{conical at infinity} if $\alpha|_L=0$ outside of a compact set.
\end{defn}

\begin{defn} Let $M$ be a Liouville manifold. If $L_1,L_2\subset M$ are exact, oriented Lagrangians which are conical at infinity (and thus not necesarily compact), then instead of the Floer cohomology, we consider the \textbf{wrapped Floer cohomology} $HW^*(L_1,L_2):=HF^*(\text{wrap}(L_1),L_2)$, where $\text{wrap}(L_1)$ is a wrapping of $L_1$.
\par In particular, $\text{wrap}(L_1)$ is defined as $\phi_1(L_1)$ for a Hamiltonian isotopy $\phi_t$ associated to a Hamiltonian $H:M\to\bbR$ which is quadratic at infinity. In other words, $H=r^2$ outside a compact set.
\end{defn}

\begin{rem} If $L_1$ and $L_2$ are compact, then $HW^*(L_1,L_2)=HF^*(L_1,L_2)$.
\end{rem}

We will now explain the wrapping process by considering the following example.

\begin{exmp}\label{Cotangent Fibers} Let $M=T^*S^1$ with usual symplectic form $\omega=dx\wedge dy$ and let $L_1$ and $L_2$ be two cotangent fibers as shown in Figure \ref{twofibers}.
\end{exmp}

\begin{figure}[h]
\includegraphics[scale=0.1]{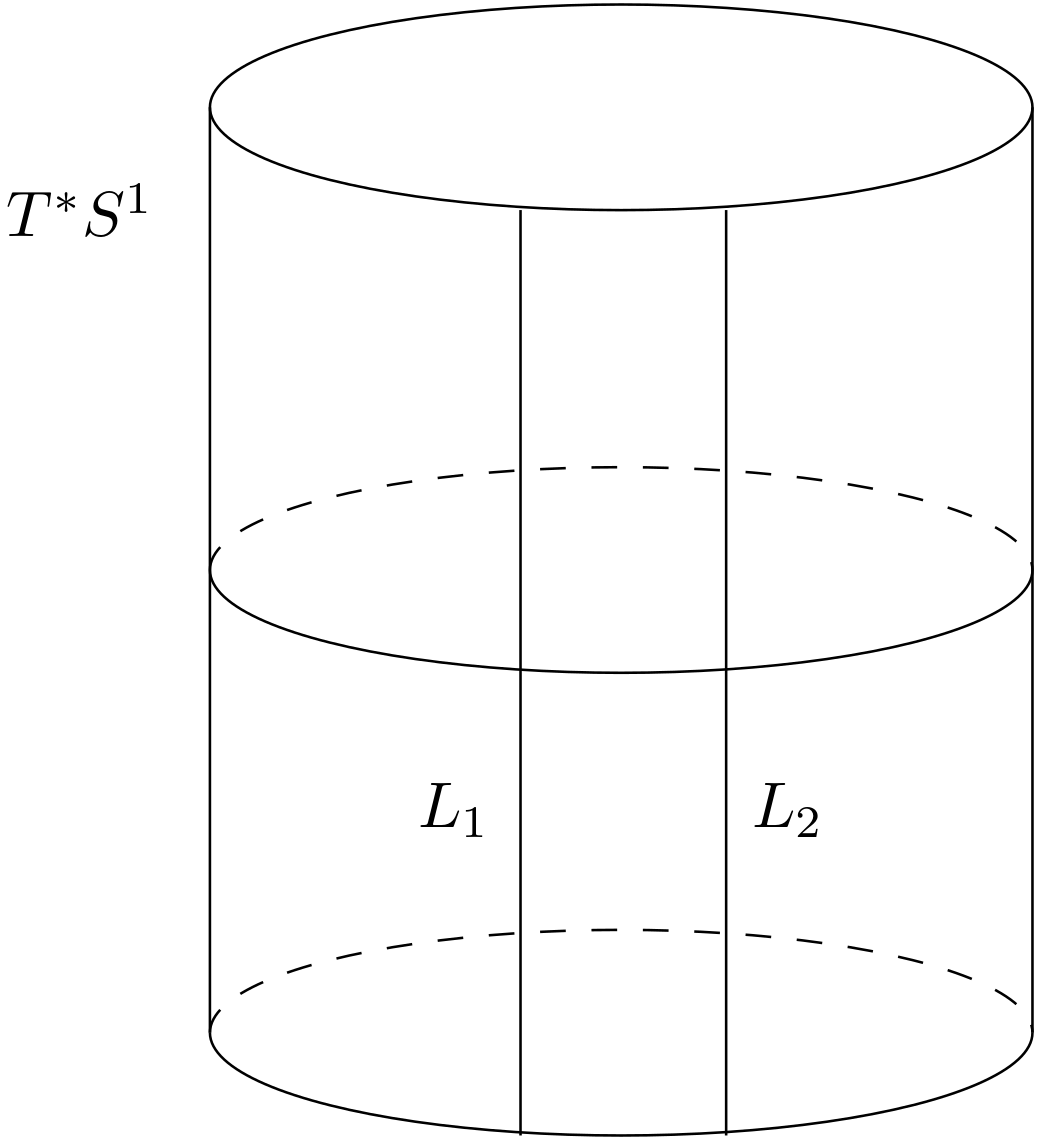}
\caption{Cotangent fibers in $T^*S^1$}
\label{twofibers}
\end{figure}

We must describe how to determine $\text{wrap}(L_1)$. We can choose the Liouville domain of $M$ to be $\overline{M}=\{(x,y)\in T^*S^1\mid\abs{y}\leq 1\}$, in which case the cylindrical end is $\partial\overline{M}\times[1,\infty)$ and $r=\abs{y}$ is the radial coordinate. Let $H=y^2$ be a Hamiltonian, noting that $H$ is quadratic at infinity, and let $X_H$ be the associated Hamiltonian vector field. Then we have
$$\omega(\cdot,X_H)=dH$$
$$dx\wedge dy(\cdot,X_H)=2y\hspace{1 mm} dy$$
which gives us $X_H=-2y\frac{\partial}{\partial x}$. Then we can define $\text{wrap}(L_1)=\phi_1(L_1)$ where $\phi_t$ is the Hamiltonian isotopy associated with our $X_H$.



\par We see that $\text{wrap}(L_1)$ and $L_2$ intersect at infinitely many points $\{x_i\mid i\in\bbZ\}$ as shown in Figure \ref{wrappedfloer}. However, we note that all these points $x_i$ have even parity, and therefore $CW^0=\bbR^\bbZ$ and $CW^1=0$. Thus we have the following chain complex
$$0\xrightarrow{0}\bbR^\bbZ\xrightarrow{0} 0$$
and the following wrapped Floer homologies
$$HW^0(L_1,L_2)=\bbR^\bbZ$$
$$HW^1(L_1,L_2)=0$$

\begin{figure}[h]
\includegraphics[scale=0.5]{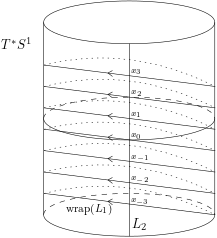}
\caption{A wrapped cotangent fiber in $T^*S^1$}
\label{wrappedfloer}
\end{figure}

\subsection{Product Structure on Floer Cochain Complex}
In the previous section, we defined a wrapped Floer cochain complex $CW^*(L_1,L_2)$ as consisting of vector spaces and differential maps between them of the form $d^i: CW^i(L_1,L_2)\to CW^{i+1}(L_1,L_2)$. These maps $d^i$ are however only one class of maps associated with wrapped Floer cochain complexes.

\begin{defn}
Let $n\in\bbN$ and let $L_1,\dots,L_{n+1}$ be Lagrangians. Then there is a class of linear maps (these maps have certain properties: see \cite{Seidel})
$$\mu^n: CW^{i_n}(L_n,L_{n+1})\otimes\dots\otimes CW^{i_1}(L_1,L_2)\to CW^{i_1+\dots+i_n+(2-n)}(L_1,L_{n+1})$$
where $\mu^n$ is determined by holomorphic disks with boundary $L_1,L_2,\dots,L_{n+1}$. We note that the differentials $d^i$ form the class of maps $\mu^1$, and that $\mu^2$ can be interpreted as composition.
\end{defn}

\begin{rem} Consider Example \ref{Cotangent Fibers}. One can work with a $\bbZ$-graded version of wrapped Floer cohomology, in which case we have $CW^0=\bbR^\bbZ$ and $CW^i=0$ for $i\neq 0$. In other words, for $n>2$, $CW^{2-n}=0$. Thus for $n>2$, the map
$$\mu^n: CW^0\otimes\dots\text{ n times }\dots\otimes CW^0\to CW^{2-n}$$
is the zero map. For $n=2$, $\mu^2$ is determined by holomorphic disks of the type shown below.
\end{rem}

\subsection{Fukaya Category}

\begin{defn}
Let $M$ be an exact symplectic manifold. Then its \textbf{Fukaya category} $\text{Fuk}(M)$ is an $A_\infty$-category consisting of compact, exact, oriented Lagrangians as objects, $CF^*(L_1,L_2)$ as morphisms, and $\mu^i$ as $A_\infty$-operations.
\end{defn}

If we want to consider non-compact Lagrangians in a Liouville manifold $M$, then we instead consider the \textbf{wrapped Fukaya category} $\text{WFuk}(M)$ with exact oriented Lagrangians which are conical at infinity as objects, $CW^*(L_1,L_2)$ as morphisms, and $\mu^i$ as $A_\infty$-operations. See \cite{Seidel} and \cite{Ganatra}.

\begin{rem} $\text{Fuk}(M)\subset\text{WFuk}(M)$ as a full $A_\infty$-subcategory.
\end{rem}

W can also consider the Fukaya category under homological perturbation (see \cite{Seidel}), which replaces $\mu^1,\mu^2\dots$ with $(\mu')^1,(\mu')^2,\dots$ where $(\mu')^i=0$ for $i> 2$. This perturbation does not affect	 $HF^*$ or $HW^*$.

\begin{defn} The $A_\infty$-categories with $\mu_3,\mu_4\dots=0$ are called \textbf{differential  graded categories}, or \textbf{dg categories}. Explicitly, a category $C$ is a dg category if its morphism spaces $\text{hom}^*(A,B)$ are ($\bbR$-linear, $\bbZ$-graded) cochain complexes with differential map $d:\text{hom}^n(A,B)\to\text{hom}^{n+1}(A,B)$ and ($\bbR$-linear) composition maps $\circ: \text{hom}^m(B,E)\otimes \text{hom}^n(A,B)\to \text{hom}^{n+m}(A,E)$ for any objects $A,B,E$ and $n,m\in\bbZ$ which satisfy the following
\begin{enumerate}
\item[1)] The degree of $\text{id}_A$ is 0 and $d(\text{id}_A)=0$ , where $\text{id}_A: A\to A$ is the identity morphism of $A$.
\item[2)] The (graded) Leibniz rule: $d(gf)=(dg)f + (-1)^{|g|}g(df)$ where $|g|$ is the degree of $g$.
\end{enumerate}
\end{defn}

\begin{defn} Let $C$ be a dg category with objects $A,B$. Then we can define $\text{Hom}^*(A,B)$ as the cohomology of the cochain complex $\text{hom}^*(A,B)$.
\end{defn}

\begin{defn} A functor $F:C\to D$ between dg categories that respects the dg structure is called a dg functor. More specifically, we want $F(dx)=d(F(x))$ for any morphism $x$ in $C$.
\end{defn}

\section{Representations of DG Categories}

\begin{defn}
We refer to the dg category whose objects are unbounded (resp. bounded) cochain complexes of $\bbR$-vector spaces as $\text{Mod }\bbR$ (resp. $\text{Mod}^b\bbR$).
\end{defn}

\begin{prop}\label{ModR}
$\text{Mod }\bbR$ is equivalent to its full dg category with graded vector spaces, i.e. cochain complexes with zero differential map, as objects.
\end{prop}

From now on, we will work with the model of $\text{Mod }\bbR$ or $\text{Mod}^b\bbR$ as given in Proposition \ref{ModR}.

\begin{defn}\label{ModC}
Let $C$ be a dg category. Then $\text{Mod }C$ is a dg category whose objects are $A_\infty$-functors from $C^{op}$ to $\text{Mod }\bbR$.
We can similarly define the dg category $\text{Mod}^b C$ using $\text{Mod}^b\bbR$ in place of $\text{Mod }\bbR$.
\end{defn}

\begin{prop}
If $C$ is semifree, i.e. compositions are free and differentials are filtered (see \cite{KarabasLee}), then we can replace $A_\infty$-functors with dg functors in Definition \ref{ModC}.
\end{prop}

\begin{defn}\label{Yoneda} If $C$ is a dg category, then we can represent it linear algebraically via the \textbf{Yoneda embedding}
\begin{align*} y\colon & C\hookrightarrow \text{Mod }C \\
& A\mapsto \text{hom}^{*}( \rule{0.3 cm}{0.1 mm} , A)
\end{align*}
\end{defn}

In the above definition, $\text{Mod }C$ consists of graded vector spaces with linear maps between them. We can use the Yoneda embedding to better understand the specific case of the cotangent bundle. 

\begin{thm}\label{Fukaya Representation}\cite{Abouzaid} Let $M$ be a manifold and $L$ be a cotangent fiber of $T^*M$. Create a dg category $C$ with object $L$, morphism $CW^*(L,L)$, and operations $\mu^1,\mu^2$ (after possibly homological perturbation). Then,
$$\text{Fuk}(T^*M)\subseteq \text{Mod}^b C$$
as a full dg subcategory where $\text{Mod}^b C$ consists of finite dimensional graded vector spaces and linear maps between them.
\end{thm}

Note that $\text{Fuk}(T^* M)$ contains all (compact, connected, exact) Lagrangians of $M$, not just the cotangent fibre $L$. In this way Theorem \ref{Fukaya Representation} allows us to ``generate" the Lagrangians of the cotangent bundle with just one cotangent fiber.

\begin{exmp}\label{sphere category}
Let $M=S^2$. If we let $L=T_x^* S^2$, we can describe its representation. We create a dg-category $C$ with object $L$, morphism $CW^*(L,L)$, and operations $\mu^1, \mu^2$. More specifically, \cite{Pinwheel} describes how $CW^*(L,L)$ is $\bbR[z]$, where $\abs{z}=-1$ and $dz=0$. Moreover, the product structure $\mu^2$ is free, meaning $z^i\otimes z^j=z^{i+j}$. Hence the category can be expressed by the following quiver diagram:

\[\begin{tikzcd}[scale=2]
	\bullet\ar[loop right,"z"]
\end{tikzcd}\]

which means the morphisms of $C$ are generated by $z$.

\par Then, the elements of $\text{Mod}^b C$ are of the following form.

\[\begin{tikzcd}[scale=2]
	V^*\ar[loop right,"f^*"]
\end{tikzcd}\]

where $V^*$ is a bounded graded vector space and $f^*$ is a graded linear map of degree $-1$. By Theorem \ref{Fukaya Representation}, such elements can be used to represent Lagrangians in $\text{Fuk}(T^*S^2)$.

\end{exmp}

\begin{rem} Maps $f$ between graded vector spaces always have differential $df=0$, as is the case in the previous example.
\end{rem}

The following theorem will help us work with maps between representations.

\begin{thm}\label{mapsbetreps}[\cite{Seidel},\cite{karabas-3}]
Let $C$ be a dg-category with the quiver representation

\[\begin{tikzcd}[scale=2]
	\bullet\ar[loop right,"x_i"]
\end{tikzcd}\]
with the arrows $x_i$ for $i=1,\dots,n$ (possibly invertible up to homotopy), and let
\begin{center}
$\overline{V}=\left(\begin{tikzcd}[scale=2]
	V^* \ar[loop right,"f^*_i"]
\end{tikzcd}\right)$

$\overline{W}=\left(\begin{tikzcd}[scale=2]
	W^* \ar[loop right,"g^*_i"]
\end{tikzcd}\right)$
\end{center}

be two objects in $\text{Mod }^b C$. Then the morphism complex $\text{hom}^*(\overline{V},\overline{W})$ is given by
$$\text{hom}^*(\overline{V},\overline{W})=\text{hom}^*(V^*,W^*)\oplus\bigoplus_{i=1}^n \text{hom}^*(V^*,W^*)[\abs{x_i}-1]$$
\par If the differentials of the morphisms in the category $dx_i=0$ for all $1=1,\dots,n$, then the differential map is given as follows. Define
$$U^*:=\text{hom}^*(V^*,W^*)$$
$$s_i^*:=g_i^* t_0^*-(-1)^{\abs{t_0^*}\abs{x_i}}t_0^* f_i^*$$
Then the differential map is given by

\begin{align*} U^*\oplus\bigoplus_{i=1}^n U^*[\abs{x_i}-1]\to & U^*\oplus\bigoplus_{i=1}^n U^*[\abs{x_i}-1]\\
(t_0^*,t_1^*,\dots,t_n^*)\mapsto & (0,s_1^*,\dots,s_n^*)\\
\end{align*}

\end{thm}

\begin{exmp}\label{zero section in S^2} We can give the specific representation for $X_0$, the zero-section of $T^* S^2$, but we must first consider the following fact. If $K$ is a compact Lagrangian, then the graded vector space $V^*$ in its representation is given by $V^*=HW^*(L,K)$, where $L$ is the cotangent fiber in the dg-category $C$. Then, we know that $X_0$ intersects $L$ at one point, so $HW^*(L,X_0)=\bbR$. Then, the zero section $X_0$, oriented appropriately, has the following representation in $\text{Mod}^b C$:

\begin{center}
$V_{X_0}=$(\begin{tikzcd}[scale=2]
	\bbR\ar[loop right,"0"]
\end{tikzcd})
\end{center}

Note that the map is zero necessarily since it is a degree $-1$ map.
\par We also recall that $X_0$ is a sphere, so by Theorem \ref{floer is homology},
$$HF^i(X_0,X_0)=H^i(X_0)=\begin{cases} \bbR & i=0 \\ 0 & i=1 \\ \bbR & i=2 \\ 0 & \text{otherwise}\end{cases}$$

We can recover this cohomology from the representation $V_{X_0}$ by considering $CF^*(X_0,X_0)=\text{hom}^*(V_{X_0},V_{X_0})$, i.e., maps between the representation and itself. By Theorem \ref{mapsbetreps}, we get
$$\text{hom}^*(V_{X_0},V_{X_0})=\text{hom}^*(\bbR,\bbR)\oplus\text{hom}^*(\bbR,\bbR)[-2]=\bbR\oplus\bbR[-2]$$

Thus the relevant chain complex is
$$CF^*(X_0,X_0)=\text{hom}^*(V_{X_0},V_{X_0})=0\rightarrow\bbR[0]\rightarrow 0\rightarrow\bbR[-2]\rightarrow 0$$
where all the maps are necesarilly the zero map. We can then recover the desired homology
$$HF^i(X_0,X_0)=\text{Hom}^i(V_{X_0},V_{X_0})=\begin{cases} \bbR & i=0 \\ 0 & i=1 \\ \bbR & i=2 \\ 0 & \text{otherwise}\end{cases}$$
\end{exmp}

\section{Lagrangians in $T^*S^2$ and $T^*T^2$}
We can use the techniques developed in this paper to characterize the orientable, compact, connected, exact, embedded Lagrangians in the cotangent bundles of the sphere and the torus.

\subsection{$T^*S^2$}
We make the following claim: any orientable, compact, connected, embedded, exact Langrangian surface in $T^*S^2$ is topologically $S^2$. We first let $L$ be a cotangent fiber and let $C$ be the category associated with $T^*S^2$ that we introduced in Example \ref{sphere category}. Then we need the following lemma.

\begin{lem}\label{concentrated} Let $S$ be a surface represented by $V_S\in\text{Mod}^b C$, i.e,
\begin{center}
$V_{S}=\left(\begin{tikzcd}[scale=2]
	V^*\ar[loop right,"f^*"]
\end{tikzcd}\right)$
\end{center}
where $V^*$ is a bounded graded vector space and $f^*$ is a degree $-1$ graded linear map. If $V^*$ is not concentrated at a single degree, that is, if $V^n, V^{n+k}\neq 0$ for some $n,k\in\bbZ$ with $k\geq 1$, then there exist $m,l\in\bbZ$ with $l\geq 4$ such that both
$$\text{Hom}^m(V_S,V_S)\neq 0$$
$$\text{Hom}^{m+l}(V_S,V_S)\neq 0$$
\end{lem}

\begin{proof} Assume $V^*$ is bounded and not concentrated at a single degree. Then there exist $n,k\in\bbZ$ with $k\geq 1$ such that $V^n,V^{n+k}\neq 0$ and $V^i=0$ for $i<n$ or $i>n+k$.Then $\hom^*(V^*, V^*)$ is not concentrated at a single degree. Define $W^i:=\text{hom}^i(V^*,V^*)$. Then we can write
$$\hom^*(V^*,V^*)=\dots0\rightarrow 0\rightarrow W^{-k}\rightarrow\dots\rightarrow W^k\rightarrow0\rightarrow0\dots$$
where:
$$W^{-k}=\text{hom}^{-k}(V^*,V^*)=\bigoplus_{i\in\bbZ}\hom^*(V^i,V^{i-k})=\hom^*(V^{n+k},V^n)\neq 0$$
We similarly have that $W^k\neq 0$, and by assumption we have $W^i=0$ for $i<-k$ and $i>k$. We then have by Theorem \ref{mapsbetreps} that
$$\hom^*(V_S,V_S)=\hom^*(V^*,V^*) \oplus \hom^*(V^*, V^*)[-2]$$
and so we have in particular that 
$$\hom^{-k}(V_S,V_S)=W^{-k}\neq 0\text{ and }\hom^{k+2}(V_S,V_S)=W^k\neq 0$$
and also
$$\hom^i(V_S,V_S)=0\text{ if } i<-k\text{ or } i>k+2$$
\par Then, because the relevant differential maps
$$d^{-k}:\hom^{-k}(V_S,V_S)\to\hom^{-k+1}(V_S,V_S)$$
$$d^{k+1}:\hom^{k+1}(V_S,V_S)\to\hom^{k+2}(V_S,V_S)$$
are zero by Theorem \ref{mapsbetreps}, we observe that
$$\text{Hom}^{-k}(V_S,V_S),\text{Hom}^{k+2}(V_S,V_S)\neq 0$$
Letting $l:=2k+2$ and $m:=-k$, we get $l\geq 4$ since $k\geq 1$, and so we have our result.
\end{proof}

We can now prove our claim from the beginning of this section.
\begin{thm}\label{sphere} Let $S$ be an orientable, compact, connected, exact Lagrangian surface in $T^* S^2$. Then $S$ is topologically a sphere $S^2$.
\end{thm}

\begin{proof} Let $V_S$ represent $S$. We know that $\dim S=2$ since it is a Lagrangian of the four-dimensional cotangent bundle. Thus, by Proposition \ref{0 homology}, we know $\text{Hom}^i(V_S,V_S)=0$ for all $i\neq 0,1,2$. Then, by Lemma \ref{concentrated}, $V^*$ must be concentrated at a single degree, i.e, $V^*=\bbR^m[s]$ for some grading $s$ and for some $m\in\bbZ$. Then, by Theorem \ref{mapsbetreps}
$$\text{hom}^*(V_S,V_S)=\bbR^{m^2}\oplus\bbR^{m^2}[-2]$$
so
$$CF^*(S,S)= 0\xrightarrow{0}\bbR^{m^2}\xrightarrow{0} 0\xrightarrow{0} \bbR^{m^2}[-2]\xrightarrow{0} 0$$
and we have the cohomology
$$HF^i(S,S)=H^i(S)=\begin{cases}  \bbR^{m^2} & i=0 \\ 0 & i=1 \\ \bbR^{m^2} & i=2 \\ 0 & \text{otherwise} \end{cases}$$
Since $S$ is connected, $H^0(S)=\bbR$ and $H^i(S)=0$ for $i<0$. Hence $m=1$. As a result, $H^*(S)$ is the cohomology of the sphere $S^2$. Since surfaces are uniquely classified by their (co)homologies, the only orientable, compact, connected, exact, embedded Lagrangians in $T^*S^2$ are spheres.
\end{proof}

\begin{rem} In fact, we can represent $S$ as
\begin{center}
$V_{S}=\left(\begin{tikzcd}[scale=2]
	\bbR[s]\ar[loop right,"0"]
\end{tikzcd}\right)
=
\left(\begin{tikzcd}[scale=2]
	\bbR\ar[loop right,"0"]
\end{tikzcd}\right)[s]$
\end{center}
since $f^*=0$ necessarily because it is of degree $-1$.
\par We can note that this is the representation of the zero section $X_0[n]$, and thus our claim is consistent with the Nearby Lagrangian Conjecture, which states that every closed exact Lagrangian in the cotangent bundle of a closed manifold is Hamiltonian isotopic to the zero section.
\end{rem}

\subsection{$T^*T^2$}
We now make the following claim about the cotangent bundle of the torus $T^2$.
\begin{thm}\label{torus}
Let $S$ be an orientable, compact, connected, embedded, exact Langrangian surface in $T^*T^2$. Then $S$ is topologically a torus $T^2$.
\end{thm}

Unlike the proof for $T^*S^2$, this proof does not require calculating (co)homologies, but rather only the Euler characteristic.
\begin{proof} If $L$ is a cotangent fiber of $T^*T^2$, then $\text{Fuk}(T^*T^2)\simeq\text{Mod}^b C$ where $C=CW^*(L,L)$. By \cite{KarabasLee}, the quiver representation of $C$ of given by

\[\begin{tikzcd}[scale=2]
	\bullet\ar[loop right,"\textit{m,n,h}"]
\end{tikzcd}\]

where $\abs{m}=\abs{n}=0$, $\abs{h}=-1$, $dm=dn=0$, $dh=mn-nm$, and $m,n$ are invertible up to homotopy.

\par Then the representations in $\text{Mod}^b C$ takes the following form

\[\begin{tikzcd} V^*\ar[loop right,"\alpha\text{, }\beta\text{, }\gamma"]
\end{tikzcd}\]
where $\abs{\alpha}=\abs{\beta}=0$, $\abs{\gamma}=-1$, $d\gamma=\alpha\beta-\beta\alpha=0$, and $\alpha,\beta$ are invertible up to homotopy (all differentials of $V^*$ are zero, so ``up to homotopy'' can be dropped).
\par Let $K$ be a compact, embedded, connected, orientable, exact Lagrangian, and let $V_K$ be its representation in $\text{Mod}^b C$. Then by Theorem \ref{mapsbetreps}
\begin{center} $\text{hom}^*(V_K,V_K)=\text{hom}^*(V^*,V^*)\oplus\text{hom}^*(V^*,V^*)[|m|-1]\oplus\text{hom}^*(V^*,V^*)[|n|-1]\oplus\text{hom}^*(V^*,V^*)[|h|-1]$

$=\text{hom}^*(V^*,V^*)\oplus\text{hom}^*(V^*,V^*)[-1]\oplus\text{hom}^*(V^*,V^*)[-1]\oplus\text{hom}^*(V^*,V^*)[-2]$
\end{center}

Then, by Proposition \ref{euler props}
\begin{center}$\chi(\text{hom}^*(V_K,V_K))=\chi(\text{hom}^*(V^*,V^*))+\chi(\text{hom}^*(V^*,V^*)[-1])+\chi(\text{hom}^*(V^*,V^*)[-1])+\chi(\text{hom}^*(V^*,V^*)[-2])$\end{center}

\begin{center}$
=\chi(\text{hom}^*(V^*,V^*))+(-1)^{-1}\chi(\text{hom}^*(V^*,V^*))+(-1)^{-1}\chi(\text{hom}^*(V^*,V^*))+(-1)^{-2}\chi(\text{hom}^*(V^*,V^*))$\end{center}
$$=(1-1-1+1)\chi(\text{hom}^*(V^*,V^*))=0$$
Then, we can also write
$$0=\chi(\text{hom}^*(V_K,V_K))=\chi(\text{hom}^*(K,K))=\chi(\text{Hom}^*(K,K))=\chi(H^*(K))$$
and since we previously noted that compact, orientable surfaces are uniquely identified by their Euler characteristics
$$\chi(H^*(K))=0\implies K=\text{Torus}$$
as desired. 
\end{proof}

\begin{rem}
One can in principle prove that for any genus $g$ surface $S$, any orientable, compact, connected, embedded, exact Lagrangian in $T^*S$ is topologically $S$ via methods similar to the ones we used in the $S=S^2$ case. However, one first needs to determine the quiver for $C$ in the equivalence $\text{Fuk}(T^*S)=\text{Mod}^b(C)$. This can be done by applying a similar computation given in \cite{KarabasLee} of the $S=T^2$ case.
\end{rem}

\bibliographystyle{plain}
\bibliography{SymplecticGeometryReferences}

\end{document}